\documentclass[11pt]{article}
\usepackage[italian, british]{babel}
\usepackage{placeins}
\usepackage{eurosym}
\usepackage{subfig}
\usepackage{url}

\usepackage{graphicx}
\usepackage{newlfont}
\usepackage{amsfonts} 
\usepackage[dvipsnames]{xcolor} 
\usepackage{amssymb}
\usepackage{amsmath}
\usepackage{latexsym}
\usepackage{amsthm}
\usepackage{listings}
\theoremstyle{plain}
\newtheorem{thm}{Theorem}[section] 

\newtheorem{prop}{Proposition}
\newtheorem{cor}{Corollary}
\newtheorem{lem}[thm]{Lemma}
\theoremstyle{definition}
\newtheorem{defn}{Definition}
\newtheorem{exmp}{Example}

\usepackage{algorithm}
\usepackage[titletoc]{appendix}

{\left\lbrace\begin{array}{@{}l@{}}}  
{\end{array}\right.}  
\usepackage{mathdots} 
\usepackage{mathtools}

\usepackage{siunitx}
\usepackage{xcolor}
\usepackage{listings}
\usepackage[framed,numbered]{matlab-prettifier}

\usepackage{tikz}
\author{Giulia Palma}
\title{A study on the fixed points of the $\gamma$ function} 

\begin{document} 

\title{A study on the fixed points of the $\gamma$ function}

\author{A.~Frosini ~\thanks{Universit\`a
di Firenze, Dipartimento di Matematica e Informatica, Viale Morgagni
65, 50134 Firenze, Italy}
\and G.~Palma ~\thanks{Universit\`a di
Siena, Dipartimento di  Ingegneria dell'Informazione e Scienze
Matematiche, Via Roma 56, 53100 Siena, Italy}
\and E. Pergola $^*$
\and S.~Rinaldi $^\dagger$}

\maketitle

\begin{abstract}

Recently a permutation on Dyck paths, related to the chip firing game, was introduced and studied by Barnabei et al. \cite{BBCC}. It is called $\gamma$-operator, and uses symmetries and reflections to relate Dyck paths having the same length.  The study of the fixed points of $\gamma$ was carried on in  \cite{BBCC}, where the authors provided a characterization of these objects, leaving the problem of their enumeration open.
In this paper, using tools from combinatorics of words, we determine new combinatorial properties of the fixed points of $\gamma$. Then we present an algorithm, denoted by \textbf{GenGammaPath}($t$), which receives as input an array $t=(t_0, \ldots ,t_{k})$ of positive integers and generates all the elements of $F_{\gamma}$ with degree $k$.

\end{abstract}

\vspace{-.4cm}

\section{Introduction}\label{intro}

The recent study of the Riemann-Roch Theorem for graphs presented in
details by Baker and Norine in \cite{BN} gave rise to a series of
side researches that overflew the main stream of graph theory,
touching combinatorics on words and theory of formal languages, as
well. An equivalent presentation can be provided in terms of the
chip firing game played on a graph $G=(E(G), V(G))$: let us consider
an initial configuration where at each vertex $v\in V(G)$ is
assigned an integer number $f_v$ of coins both positive, and
negative. The objective of the game is to reach a positive
configuration $f=(f_1,\dots,f_n)$, i.e. a configuration where all
the vertices have a nonnegative number of coins, by using a sequence
of two types of moves that, for each vertex, consist either in
borrowing one coin from each of its neighbors or in giving one coin
to each of its neighbors. Such a configuration is called a {\em
winning configuration}, and each sequence of moves which leads to
such a configuration is a {\em winning strategy}. Let
$g=|E(G)|+|V(G)|-1$, Theorem~$1.9$ in \cite{BN} states that

\begin{thm}\label{teo:karin}
Let $N = deg(D)$ be the total number of coins present in the graph
$G$ at any step of the game. \item{1.} If $N \geq g$, then there is
always a winning strategy. \item{2.} If $N < g$, then there is
always an initial configuration for which no winning strategy
exists.
\end{thm}

In~\cite{cori}, the authors considered a restriction of the game to
complete graphs, and they studied the notion of rank of the winning
configurations. Basing on Theorem \ref{teo:karin}, they provided a
useful characterization of the winning configurations in terms of
parking configurations. Let us recall that, in a complete graph, a
sequence $(f_1,\dots,f_{n-1})$ is the initial sequence of a parking
configuration if and only if, after reordering, it is a weakly
increasing sequence $(g_1,\dots,g_{n-1})$ such that for each $i$ it
holds $g_i<i$.


In this paper we are concerned with an alternative characterization
of the winning configurations of complete graphs as fixed points of
an operator, called $\gamma$-operator, that geometrically acts on
Dyck paths \cite{BBCC}. As one can expect, Dyck paths may be related
to parking configurations in the sense that each length $k$ prefix
of their coding words contains a number of descending steps that is
upper bounded by $k-1$.

The authors of \cite{BBCC} present three permutations on Dyck words. The first one, $\alpha$, is related to the Baker and Norine theorem on graphs, the second one, $\beta$ is the symmetry, and the third one, $\gamma$, is the composition of these two. The fixed points of $\alpha$ and $\beta$ are not difficult to characterize, and the studies of \cite{BBCC} concentrate on the characterization of the fixed points of $\gamma$, showing combinatorial properties of its cycles.

In Section 2 we recall basic definitions and known results, mainly from \cite{BBCC}. In Section 3 we provide a new characterization of the fixed points of $\gamma$,  in terms of combinatorial properties (and in particular some recursive decomposition) of the words encoding them. The objective of this characterization is that of obtaining an algorithm for 
the generation of these fixed points. In Section 4, we first define the notion of {\em degree} of a fixed point of $\gamma$, then we write down an algorithm, denoted by \textbf{GenGammaPath}($t$), which receives as input an array $t=(t_0, \ldots ,t_{k-1})$ of positive integers and generates all the elements of $F_{\gamma}$ with degree $k$.

In the final section we investigate the relation between the degree and the length of an element of $F_\gamma$. We believe that the algorithm \textbf{GenGammaPath} can be used in some further research for exhaustive generation of the fixed points of $\gamma$, and also to study the generating function of these objects according to their length.

\vspace{-.4cm}

\section{Definitions and preliminary results}\label{sec:intro}

Let $w$ be a word on the free monoid $\Sigma^\star$, where
$\Sigma=\{a,b\}$. As usual, let $|w|$ denote the length of $w$, i.e.
the number of its letters, and let $|w|_a$
and $|w|_b$ denote the number of the occurrences of the letters $a$
and $b$ in $w$, respectively. Furthermore, to each word $w\in
\Sigma^\star$, we associate the integer number
$\delta(w)=|w|_a-|w|_b$. A word $u$ is a {\em prefix} of $w$ if, for
some $v$ we have $w=uv$; in this case $v$ is said to be a {\em
suffix} of $w$. Two words $w$ and $w'$ are {\em
conjugate} if there are words $u,v$ such that $w=uv$ and $w'=vu$.

Dyck words are an almost ubiquitous family of words which show
natural connections with a huge quantity of problems in different
scientific areas: more importantly for us, in
\cite{cori} it is shown a strict connections between Dyck words and
parking configurations on complete graphs.

\begin{defn}
A word $w\in \Sigma^\star$ of length $2n$ is a Dyck word if and only
if $\delta(w)=0$ and, for each prefix $v$ of $w$, it holds
$\delta(v)\geq 0$.
\end{defn}

\noindent By definition, for each Dyck word $w$ of length $2n$, it
holds $|w|_a=|w|_b=n$.

Let us consider the two sets of words $A_n$ and $D_n$ defined as
follows: $A_n$ contains any word $w$ of length $2n+1$ such that
$|w|_a=n$ and $|w|_b=n+1$, while the set $D_n$ is the set of Dyck
words followed by a single occurrence of $b$. By definition, we have
that, with $n>0$, $D_n\subset A_n$. A non trivial connection between
these two sets is established by the so called {\em Cycle Lemma},
illustrated in \cite{Dvore}, which can be stated as follows:

\begin{lem}{\bf (Cycle Lemma)}\label{lem:cyclic}
Let $w$ be an element of $A_n$. Then $w$ admits a unique
factorization $w=uv$ such that the conjugate word $w'=vu$ belongs to
$D_n$.
\end{lem}

The Cycle Lemma states that the conjugacy relation induces a
partition on $A_n$ into equivalence classes whose minimal
lexicographical representatives are exactly the elements of $D_n$.

\subsection{Three permutations on the set $D_n$}

In this section, unless otherwise specified, we borrow notation and
definitions from \cite{BBCC}. We present two involutions $\alpha$
and $\beta$ on $D_n$, whose composition gives a permutation, called
$\gamma$, on which our study will be focused.

Few more definitions are needed: given a word $w=w_1 w_2 \dots w_m$, its {\em complement}
$\overline{w}$ is the word $\overline{w}_1\overline{w}_2\dots
\overline{w}_m$, where $\overline{w}_i$, with $1\leq i \leq m$,
exchanges the letter $a$ with $b$ and viceversa, its {\em
mirror} $\widetilde{w}$ is the word $w_m \dots w_2 w_1$, and its symmetric $Sym(w)$ is the complement of its mirror, i.e. the word $\overline{w}_{m} \dots \overline{w}_1$.

\medskip

\noindent \textit{The involution $\alpha$.} The function $\alpha$, introduced in \cite{CL}, maps an element $w$
of $D_n$ onto the unique conjugate of $\widetilde{w}$ that belongs
to $D_n$. By the cycle lemma, we know that there is exactly one such element. As an example, let the word $w=aabbaababaabbbb$ be an element of $D_n$. By definition,  $\widetilde{w}=bbbbaababaabbaa$ and its
conjugate is $\alpha(w)= aababaabbaa \:\: bbbb.$ In \cite{BBCC} the authors show that $\alpha$ is indeed an
involution on $D_n$.

\medskip

\noindent \textit{The involution $\beta$.} The function $\beta$ maps any element $w=w_1 w_2 \dots w_{2n+1}$ of
$D_n$ onto the word obtained by applying the symmetry operator to its first $2n$ letters, i.e.
$\beta(w)=Sym(w_1\dots w_{2n}) \: \overline{w}_{2n+1}.$ The fact that $\beta$ is an involution is immediate, since
it realizes the central symmetry of the first $2n$ elements of $w$,
as one can check by considering $w=aabbaababaabbbb\in D_n$. Then
$\beta(w)= \overline{bbbaababaabbaa}\: b=aaabbababbaabb\:\:b$
that still belongs to $D_n$.

\medskip

\noindent \textit{The permutation $\gamma$.}
Let $w$ be a word of $D_n$. The {\em principal prefix} (resp. {\em
principal suffix}) of $w$ is the shortest prefix (resp. suffix) $u$
of $w$ such that $\delta(u)$ is maximal.
Now, the mapping $\gamma$ is defined as the composition of $\alpha$
and $\beta$. Formally, with $w\in D_n$, we have:
$\gamma(w)=\alpha(\beta(w)).$
By definition, $\gamma$ acts on a word $w=u \:\: b$ of $D_n$, and
provides the unique word $w'$ in $D_n$ that is the conjugate of
$\overline{u}\: b$. It is easy to check that, if $w=u\: v \:b$, then
$\gamma(w)=\overline{v}\: b\: \overline{u}$, with $u$ being the
principal prefix of $w$.
The application of $\gamma$ to the word
$w=aabbaababaabbbb$ gives
$$\gamma(w)= \alpha(\beta(w))=\alpha(aaabbababbaabb\: b)=aaa \: bbbaabbababb.$$
The mapping $\gamma$ is a permutation of the words of
$D_n$. Actually $\gamma$ determines a partition of the words of
$D_n$ into classes, or {\em cycles}, that contain all the words that
can be obtained by iterated applications of $\gamma$.
Again in \cite{BBCC}, it was proved that each cycle induced by $\gamma$ has
odd cardinality; on the other hand, there are no results concerning
enumeration of the elements of the cycles with respect to their
length.

\smallskip

Dyck words can be naturally represented as lattice paths commonly
known under the name of {\em Dyck paths}. They are paths in the
first quadrant which begin at the origin, end at $(0,2n)$ and use
North-East steps $(1,1)$ ({\em rise} steps) and South-West steps
$(1,-1)$ ({\em fall} steps). The correspondence between Dyck words
and Dyck paths is obtained coding rise (resp. fall) steps with the
letter $a$ (resp. $b$). To understand the coding, see for instance the example in Fig.~\ref{fix_points},
which depicts the path associated with a word in $D_n$.
In a Dyck path, the {\em level} of a point of the path is its
ordinate; furthermore we call {\em peak} and {\em valley} any
occurrence in the related word of the sequence $ab$ and $ba$,
respectively.
We observe that the three mappings $\alpha$, $\beta$ and $\gamma$
defined above can be easily described in a graphical way using the
path representation of Dyck words. Such a representation also helps
us check some of the properties of the fixed points of $\alpha$,
$\beta$ and $\gamma$, which will be studied in this paper. So, from
now on, we will use the word representation and the path
representation of Dyck words indifferently.

\subsection{The fixed points of $\alpha$ and $\beta$}\
The following non trivial property, proved in \cite{BBCC}, provides
the characterization of the fixed points of $\alpha$:

\begin{prop} \label{2pal}
The word $w\in D_n$ is a fixed point for $\alpha$ if and only if $w$
is the concatenation of two palindromes. Moreover $w$ has a unique
decomposition as concatenation of two palindromes.
\end{prop}

Figure~\ref{fix_points}, $(a)$ shows a fixed point of $\alpha$; the
(unique) factorization of the path in two palindromes is pointed
out.
Note that, for each $w\in D_n$, the only conjugate $w'\in D_n$ of
$\widetilde{w}$ is such that $\widetilde{w}=uv$, $w'=vu$ and the
point of $u$ with the lowest ordinate is precisely its last point.
Concerning the involution $\beta$, each one of its fixed points can
be represented by a path $w=w_1 \dots w_{2n} \: b$ of $D_n$ such
that the prefix of length $2n$ is vertically symmetric, i.e.
$w=w_1\dots w_n \: Sym(w_1 \dots w_n) \:b$, as shown in
Fig.~\ref{fix_points}~$(b)$.

\begin{figure}[!h]
\begin{center}
\includegraphics[scale=0.7]{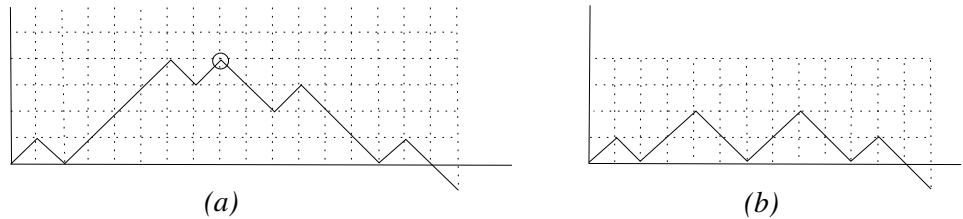}
\caption{\small Two examples of fixed points of the involutions $\alpha$, $(a)$,
and $\beta$, $(b)$. The decomposition in two palindromes
is highlighted.}\label{fix_points}
\end{center}
\end{figure}

So, let us consider the class $F_n\subset D_n$ of the fixed points
of $\gamma$ of length $2n+1$, and let $F_\gamma=\cup_{n=1}^\infty
F_n$. 


In what follows we will indicate with $w_i$ the generic path (in the case in which $w_i$ indicates the i-th component of the word it will be properly specified).
%
%
%

\section{Combinatorial properties of the fixed points of~$\gamma$}

We observe that a generic element $w\in F_\gamma$,
being a fixed point of $\beta$, can be written as $w= u \: a \: v
\: b\: {Sym(u)} \: b$ with $u\: a$ being its principal prefix, and
$b\: Sym(u) \: b$ its principal suffix.

\begin{prop}\label{prop:form}
Let $w=u \: a \: v \: b\: Sym(u) \: b$ be an element of $F_n$; the
following statements hold:
\begin{description}
\item{$(i)$} $\delta (u \: a) = M$, where $M$ is the maximal level of $w$, and $\delta (v) = 0$;

\item{$(ii)$} $u$ and $u\: a \: v$ are palindromes;

\item{$(iii)$} if $v=\lambda$, the empty word, then $w= a^n \: b^{n+1}$,
where the power notation $x^y$ stands for the usual concatenation of the string $x$ with itself $y$ times;

\item{$(iv)$} if $v\not = \lambda$, then $v$ can be decomposed as
$v=v_1 \: a\: v_2$, with $v_1$ and $v_2$ palindromes, and $v_2$
possibly empty. Furthermore, it holds $u=v_2 \: (a \: v)^t $,
with $t\geq 0$.
\end{description}
\end{prop}

\proof The proofs of $(i)$ and $(ii)$ are straightforward, and
follow from the fact that $w$ is a fixed point of both $\beta$ and
$\alpha$.
To prove $(iii)$ we observe that $v=\lambda$ implies that $u \:a$
and $u$ are both palindromes, by statement $(ii)$, so $u=a^{n-1}$.
Finally, the properties in $(iv)$ are direct consequences of the palindromicity of $u$ and $u \:a \: v$: let $v \not= \lambda$, it holds
\begin{description}
  \item[$u\: a \: v  =$ ] $\tilde{v}\:a \: \dots $, by palindromicity of $u\:a\:v$;
  \item[$u\: a \: v  =$ ] $\dots \: a\: v \:a \: v $, by palindromicity of $u$;
  \item[$u\: a \: v  =$ ] $\tilde{v}\:a \: \tilde{v}\: a \dots$, by palindromicity of $u\:a\:v$;
  \item[$\vdots \qquad \qquad \vdots \qquad \qquad \vdots$ ]
  \item[$u\: a \: v  =$ ]  $(\tilde{v}\:a)^t \: \tilde{v}_2 = v_2 \: (a\:v)^t$.
\end{description}
for some decomposition of $v$ as $v_1 \: a \: v_2$, with $v_2$
possibly empty. Hence, the two different ways of decomposing $u\:a
\:v$ directly lead to $v_2=\widetilde{v}_2$. Finally, by the
symmetry of $v$, we also deduce that $v_1=\widetilde{v}_1$, as
desired.  \qed


\medskip

\noindent{\bf Remark 1.} Let us now analyze the form of a generic element of $F_n$, with
$m_1$ (resp. $m_2$) being the level of the last (resp. first) step of
$v_1$ (resp.$v_2$).
Since, by assumption, the first and the last points of $v$ are at
the same maximal level $M$, and $v_1$ and $v_2$ are palindromes,
then the levels of the points of $v_1$ lie between $m_1$ and $M$,
and those of $v_2$ between $m_2$ and $M$. Observe in fact that, if a
step of $v_1$, resp. $v_2$, was below $m_1$, resp. $m_2$, then, due
to palindromicity, the symmetric step is above $M$, which is a
contradiction. As a consequence $m_1=m_2+1$, since $v_1$ and $v_2$
are connected by a rise step $a$.
We recall that $v$ is symmetric since it is the central part of $w$,
that is a fixed point of $\beta$. Two cases arise: either
$Sym(v_1)$ is a suffix of $v_2$ or viceversa. In the first case,
$v_2$ reaches the level $m_1$ that is lower than $m_2$ and this is a
contradiction, so necessarily we have that $Sym(a \: v_2)$ is a
prefix of $v_1$. \qed

\begin{cor}\label{coro}

Let $w=u \: a \: v \: b\: Sym(u) \: b$ be an element of $F_n$. Then the word $u a$ can be decomposed into two factors
$u_1$ and $u_2$, such that either $u_2\:\bar{u_1}$ or $\bar{u_2}\:u_1$ is palindrome.
\end{cor}

The proof follows acting as in Proposition~\ref{prop:form}, after observing that the roles of $u$ and $v$ can be exchanged and a suitable decomposition can be obtained by alternating palindromicity of $u \: a \:v$ and symmetry of $u \: a \: v \: b\: Sym(u)$. 

As an example, the path in Fig.~\ref{esempio4} has a decomposition of the form  $(u\: a \: \bar{u} \: b)^t \: u \: a \: Sym(\tilde{u}_1) \: \tilde{u}_1 \: (b \:\bar{u} \: a \:u)^t$, with $t=0$. So, it turns out that $\tilde{u}_1 = ba$, $Sym(\tilde{u}_1)=\bar{u}_1=ba$, $u_2=abaababaa$ and the sub-word $u_2 \: \bar{u}_1= abaababaa \: ba$ is palindrome.

\begin{figure}[!h]
\begin{center}
\includegraphics[scale=0.8]{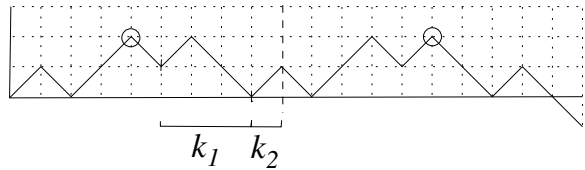}
\caption{\small A fixed point of $\gamma$ where $u=v_2$, and the decomposition of  Corollary~\ref{coro}.}\label{esempio4}
\end{center}
\end{figure}

\vspace{-1cm}

\section{The algorithm for the construction of $F_{\gamma}$ elements}

As a consequence of Corollary \ref{coro}, a generic Dyck word $w$ is an element of $F_{\gamma}$ if and only if 
\begin{enumerate}
\item $w$ is symmetric;
\item $w \: b$ is uniquely decomposable in two palindromes $\pi_1$ and $\pi_2$.
\end{enumerate}
A direct consequence of $(i)$ of Proposition \ref{prop:form} is that the decomposition of  $w \: b$ in the two palindromes $\pi_1$ and $\pi_2$ can be obtained by first defining the points $M$ and $M'$, where $M$ (resp. $M'$) is the leftmost (resp. rightmost) point of $w$ with greatest ordinate. \\ 
The path running from $(0,0)$ to $M'$ is $\pi_1$, while the one from $M'$ to the end of the path is $\pi_2$. \\
Figure \ref{fig:figura1} shows a path where the points $M$, $M'$, and the paths $\pi_1$ and $\pi_2$ have been highlighted.\\
\\
\begin{figure}[ht]
\begin{center}
\includegraphics[height=5cm] {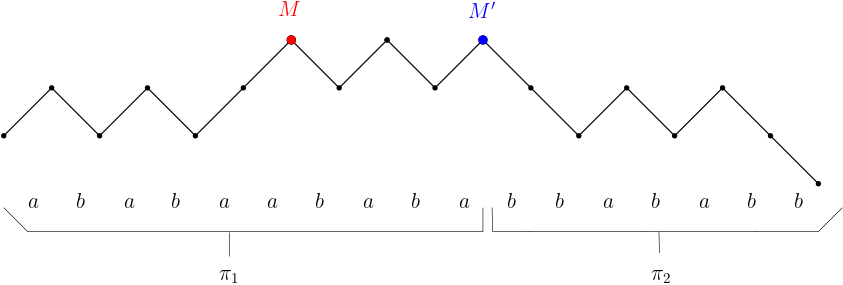}
\caption{A path where $M$, $M'$, $\pi_1$ and $\pi_2$ have been highlighted.}
\label{fig:figura1}
\end{center}
\end{figure}

Our aim in this section is to provide a recursive description of the elements of $F_{\gamma}$.
Given $w \in F_{\gamma}$ we first show that the following Proposition holds:
\begin{prop}\label{prop:decomposizione_cammino}
Let $w \in F_{\gamma}$ and $M$ and $M'$ defined as before. Then $w$ can be uniquely decomposed as $x \: z \: Sym(x)$, where z is the subpath running from $M$ to $M'$ and $\widetilde{z} \in F_{\gamma}$. \\
\\
\begin{figure}[h]
\centering
\includegraphics[height=3cm] {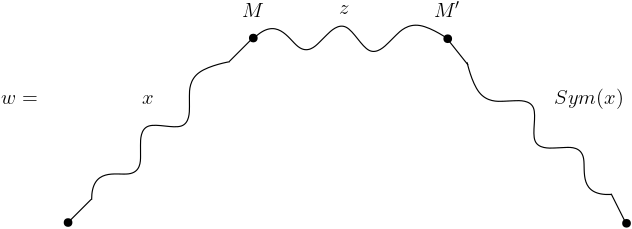}
\caption{Decomposition of the path $w$}
\label{fig:figura2}

\end{figure}
\end{prop}

\begin{proof}

To prove that $\widetilde{z} \in F_{\gamma}$ is equivalent to showing that:
\begin{enumerate}
\item $z$ is symmetric,
\item $bz$ can be uniquely decomposed in two palindromes.
\end{enumerate}
The first claim immediately follows from the fact that $w \in F_{\gamma}$ and therefore it is symmetric. Indeed $M_1$ and $M_2$ are two points one symmetric to the other and then the subpath whose extremes are $M_1$ and $M_2$ is symmetric. \\
Let us now prove $2$. Since $z \neq \lambda$ (i.e. $\lambda$ is the empty path) and $w \in F_{\gamma}$, then it follows from $(iv)$ of Proposition \ref{prop:form} that $z$ can be decomposed as $z=v_1 \: a \: v_2$, $v_2$ possibly empty and $v_1$ and $v_2$ are palindromes. \\
Since $v_2$ is palindrome, $\widetilde{v_2}$ is palindrome too and then $\pi_2 = b \: \widetilde{v_2} \: b$ is palindrome. Moreover, $\pi_1=\widetilde{v_1}$ is palindrome, since $v_1$ is a palindrome as well. Therefore $\pi_1$ and $\pi_2$ are palindromes. \\
\begin{figure}[h]
\centering
\includegraphics[height=7.8cm] {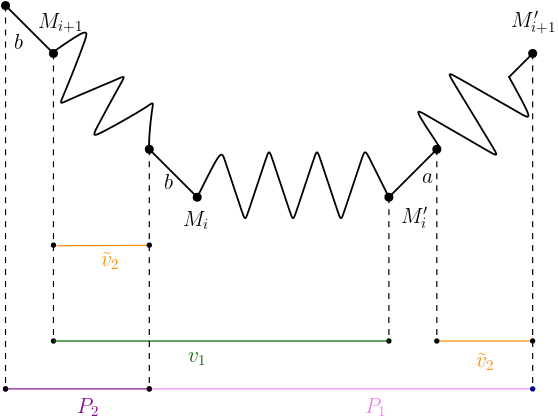}
\caption{A part of a path}
\label{fig:figura3}
\end{figure}
\end{proof}

We present an example of the previous decomposition.
\begin{exmp}
The path $w$ represented in Figure \ref{fig:figura1} can be uniquely decomposed in $xzy$, where
\[x= ababaa\]
\[z=baba\]
\[y=bbababb\]
and $\beta(z)=baba \in F_{\gamma}$.
\end{exmp} 
We now describe the following algorithm that generates all and only  the paths of $F_{\gamma}$:\\

\medskip 

\noindent \textbf{GenGammaPath}($t$)\\
\\
\textbf{Input:} an array $t=(t_0,\dots, t_{n})$ of nonnegative integers, with $t_0>0$.\\
\textbf{Output:} a path $w \in F_{\gamma}$ of degree $n$.\\
\\
\textbf{Part A} Let $n={2m}$.\\ \textbf{Initial step:} Let $u_0 = a^{t_0 -1}$ and $w_0 = a^{t_0} \: b^{t_0}$ \\
\textbf{For} $i = 1, 2, \dots, 2m$ \textbf{(Basic step):}
\begin{itemize}
\item If $i$ is even then:\\
$u_i = Sym(u_{i -1})(a \: w_{i -1})^{t_i}$ and $w_i = u_i \: a \: w_{i -1} \: b \: Sym(u_i)$;
\item If $i$ is odd then:\\
$u_i = Sym(u_{i-1})(b \: w_{i-1})^{t_i}$ and $w_i = u_i \: b \: w_{i-1} \: a \: Sym(u_i)$;
\end{itemize}
The algorithm returns the word  $w_{2m}$, and $2m$ is said the \textit{degree of w}.\\
\\
\textbf{Part B} Let $n=2m+1$.\\ \textbf{Initial step:} Let $u_0 = b^{t_0-1}$ and $w_0 = b^{t_0} \: a^{t_0}$\\
\textbf{For} $i = 1, 2, \dots , 2m + 1$ \textbf{(Basic step):}
\begin{itemize}
\item If $i$ is odd then:\\
$u_i = Sym(u_{i-1})(a \: w_{i-1})^{t_i}$ and $w_i = u_i \: a \: w_{i-1} \: b \: Sym(u_i)$
\item If $i$ is even then:\\
$u_i = Sym(u_{i-1})(b \: w_{i-1})^{t_i}$ and $w_i = u_i \: b \: w_{i-1} \: a \: Sym(u_i)$
\end{itemize}
The algorithm returns the word  $w_{2m+1}$, and $2m+1$ is said the \textit{degree of w}.\\
\\

We present an example of application of \textbf{GenGammaPath}($t$) to generate a path of degree $2$.

\begin{exmp}
Let us consider the path depicted in Figure \ref{fig:figura_camminoesempioalg}: \\
\begin{figure}[ht]
\begin{center}
\includegraphics[height=2cm] {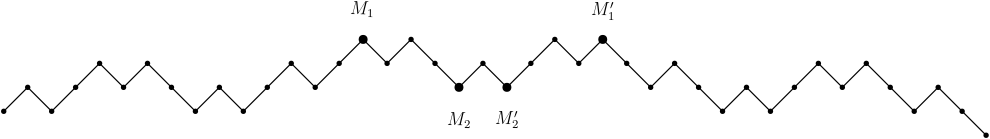}
\caption{A path $w \in F_{\gamma}$ and its decomposition as $w_0$, $w_1$, $w_2$}
\label{fig:figura_camminoesempioalg}
\end{center}
\end{figure}
Let us show how to generate $w$ using the algorithm \textbf{GenGammaPath}($t$), with $t=(1,1,1)$. 

\begin{description}
\item{Step 0.} Since $t_0=1$, then $u_0=\lambda$, and $w_0=a^{t_0}b^{t_0}=ab$\, .\\
\begin{figure}[ht]
\begin{center}
\includegraphics[height=1cm] {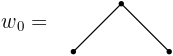} 
\qquad
\includegraphics[height=2.4cm] {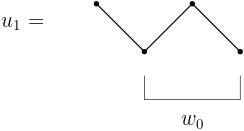}
\caption{The path $w_0$, on the left, and the path $u_1$, on the right.}
\label{fig:figuraw0}
\end{center}
\end{figure}
\item{Step 1.}  Since $t_1=1$, then:\\
$u_1=Sym(u_0)\:(b\:w_0)^{t_1}=\lambda(b \: w_0)^{t_1} =(bab)^{t_1}=bab$\\
$w_1=u_1 \: b \: w_0 \: a \: Sym(u_1)= bab \:b a b a \: Sym(bab)= babbabaaba$.\\
\begin{figure}[ht]
\begin{center}
\includegraphics[height=3cm] {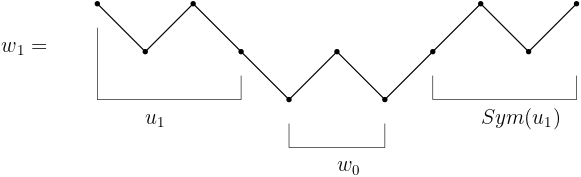}
\caption{The factorization of $w_1$ in $w$.}
\label{fig:figuraw1}
\end{center}
\end{figure}
\item{Step 2.} Since $t_2=1$, then
$$\begin{array}{l}
u_2=Sym(u_1)\:(a\:w_1)^{t_2}= Sym(bab) (ababbabaaba)^{t_2} = abaababbabaaba\\
w_2=u_2 \: a \: w_1 \: b \: Sym(u_2)=abaababbabaabaababbabaababbabbabaababbab.\\
\end{array}
$$

\begin{figure}[ht]
\begin{center}
\includegraphics[height=3cm] {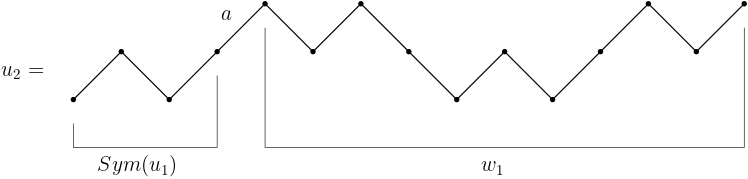}
\caption{The factorization of $u_2$ in $w$.}
\label{fig:figura1}
\end{center}
\end{figure}

\begin{figure}[ht]
\begin{center}
\includegraphics[height=3cm] {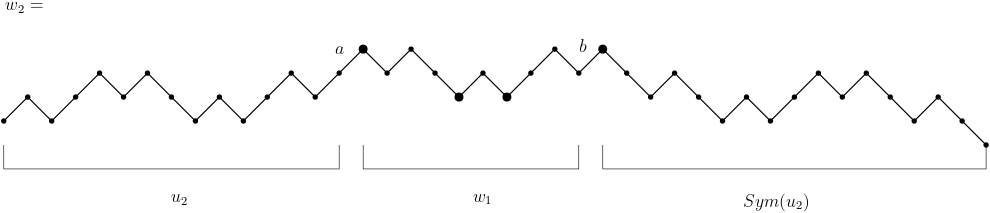}
\caption{The factorization of $w_2$ in $w$.}
\label{fig:figura1}
\end{center}
\end{figure}
Therefore, $w=abaababbabaabaababbabaababbabbabaababbabb$.\\
\end{description}
\end{exmp}

The following theorem holds:
\begin{thm}
A path $w \in F_{\gamma}$ if and only if it can be obtained by the \textbf{GenGammaPath}($t$).
\end{thm}

\begin{proof}
First of all we will prove that all the paths generated by \textbf{GenGammaPath}($t$) are elements of $F_{\gamma}$, then we will show that all elements of $F_{\gamma}$ are effectively generated by the \textbf{GenGammaPath}($t$).\\
$(\Leftarrow)$ Given an element $w$ generated by \textbf{GenGammaPath}($t$), we prove by induction on the degree $i$ of $w$ that it belongs to $F_{\gamma}$. \\

\medskip

\noindent\textbf{Base step:} If $i=0$, then $w_0=a^{t_0}b^{t_0}$,  $t_0\geq 1$. \\


The path $w_0$ is trivially symmetric and $w_0 \: b$ is uniquely decomposable into two palindromes $\pi_1$ and $\pi_2$, hence $w_0 \in F_{\gamma}$.\\

\noindent\textbf{Inductive step:} Let us assume that our claim holds for $i = n$ and let us show that it also holds for $i = n + 1$. \\
So, let $w_{n+1}$ be a path of degree $n+1$ produced by \textbf{GenGammaPath}($t$). To prove that it belongs to $F_{\gamma}$ we have to prove that:
\begin{enumerate}
\item  $w_{n+1}$ is symmetric,
\item  $w_{n+1} \: b$ is uniquely decomposable into two palindromes $\pi_1$ and $\pi_2$.
\end{enumerate}
Since the case in which $n+1$ is odd is identical to the case in which $n+2$ is even (we only have to exchange $a$ and $b$ and vice versa), we will treat, without loss of generality, only the case in which $n + 1$ is even. \\
By our construction, the path produced by \textbf{GenGammaPath}($t$) is 
\[w_{n+1}=u_{n+1} \; a \; w_n \; b \; Sym(u_{n+1})\, .\]
Moreover, we are assuming that it holds for $n$, then by the inductive argument we have that $w_n \in F_{\gamma}$. 
The we have:
\begin{enumerate}
\item $w_{n+1}$ is trivially symmetric since:
\[w_{n+1} = u_{n+1} \: a \: w_n \: b \: Sym(u_{n+1})\, ,\]
where $u_{n+1} \: a$ and $b \: Sym(u_{n+1})$ are symmetric to each by construction other and $w_n$ is symmetric since $w_n \in F_{\gamma}$. Therefore $w_{n+1}$ is symmetric since it is obtained from a symmetric path by attaching two symmetric pieces to the extremities.
\item Now we show that $w_{n+1} \; b$ is decomposable in two palindromes (for simplicity let us indicate $n+1$ by $\ell$). More precisely, we will show that this decomposition is obtained as $\pi_1$ and $\pi_2$, where $\pi_1 = u_{\ell} \: a \: w_{\ell-1}$ and $\pi_2 = b \: Sym(u_{\ell}) b$.

Let us start showing that $\pi_2$ is palindrome. We have that:
 
$\pi_2= b \: Sym(u_{\ell}) b$ is palindrome if and only if 

$Sym(u_{\ell})$ is palindrome if and only if 

$u_{\ell}$ is palindrome if and only if 

$Sym(u_{\ell-1})\: (a \: w_{\ell-1})^{t_{\ell}}$ is palindrome. \\

Therefore, let us show that $Sym(u_{\ell-1})\: (a \: w_{\ell-1})^{t_{\ell}}$ is palindrome. \\
We know that the building process has started with $w_0=a^{t_0}b^{t_0}$ and we assume that $\ell>0$ is even, then $\ell-1$ is odd. \\
We recall thar $w_{\ell-1}=u_{\ell-1} \: b \: w_{\ell-2} \: a \: Sym(u_{\ell-1})$.
By inductive hypothesis $w_{\ell-1} \in F_{\gamma}$ then by $(ii)$ of Proposition \ref{prop:form}, $u_{\ell-1}$ and $u_{\ell-1} \: b \: w_{\ell-2}$ are palindromes. Let us consider some small values of $t_\ell$:
If $t_{\ell}=0$, then 
\[u_{\ell} = Sym(u_{\ell-1}) \: (a \: w_{\ell-1})^0= Sym(u_{\ell-1}).\]
Since $u_{\ell-1}$ is palindrome, also $Sym(u_{\ell-1})$ is palindrome. \\
Let us consider some small values of $t_{\ell}$. 
If $t_{\ell}=1$, then 
\begin{eqnarray*}
u_{\ell} &=& Sym(u_{\ell-1})(a \: w_{\ell-1}) \\
 &=& Sym(u_{\ell-1}) \: a \: u_{\ell-1} \: b \: w_{\ell-2} \: a \: Sym(u_{\ell-1}).
 \end{eqnarray*}
 
Since $Sym(u_{\ell-1})$ is palindrome (indeed $u_{\ell-1}$ is palindrome) and $u_{\ell-1} \: b \: w_{\ell-2}$ is palindrome. \\
If $t_{\ell}=2$, then 
\begin{eqnarray*}
u_{\ell} &=& Sym(u_{\ell-1}) \: (a \: w_{\ell-2})^2 \\
&=& Sym(u_{\ell-2}) \: (a \: w_{\ell-1}) \: (a \: w_{\ell-1}) \\
&=& Sym(u_{\ell-2}) \: a \: u_{\ell-1} \: b \: w_{\ell-2} \: a \: Sym(u_{\ell-1}) \: a \: u_{\ell-1} b \: w_{\ell-2} \: a \: Sym(u_{\ell-1}).
\end{eqnarray*}
The latter is palindrome indeed $Sym(u_{\ell-1})$ is palindrome (since $u_{\ell-1}$ is palindrome) and $u_{\ell-1} \: b \: w_{\ell-2}$ is palindrome. \\
In general, if $\ell$ is even: 
\[u_{\ell} = Sym(u_{\ell-1}) \: a u_{\ell-1} \: b \: w_{\ell-2} \: a \: Sym(u_{\ell-1}) \: a \: u_{\ell-1} \: b \: w_{\ell-2} \: \dots a ,\]
\[Sym(u_{\ell-1}) \: a \dots u_{\ell-1} \: b \: w_{\ell-2} \: a \: Sym(u_{\ell-1}) \: a \: u_{\ell-1} \: b \: w_{\ell-2} \: a \: Sym(u_{\ell-1}).\]
The latter is palindrome indeed $Sym(u_{\ell-1})$ is palindrome (since $u_{\ell-1}$ is palindrome) and $u_{\ell-1} \: b \: w_{\ell-2}$ is palindrome. \\
Instead, if $\ell$ is odd: \\
\[u_{\ell} = Sym(u_{\ell-1}) \: a \: u_{\ell-1} \: b \: w_{\ell-2} \: a \: Sym(u_{\ell-1}) \: a \: u_{\ell-1} \: b \: w_{\ell-2} \dots u_{\ell-1} \: b,\]
\[w_{\ell-2} \dots u_{\ell-1} \: b \: w_{\ell-2} \: a \: Sym(u_{\ell-1}) \: a \:u_{\ell-1} \: b \: w_{\ell-2} \: a \: Sym(u_{\ell-1}).\]
The latter is palindrome indeed $Sym(u_{\ell-1})$ is palindrome (since $u_{\ell-1}$ is palindrome) and $u_{\ell-1} \: b \: w_{\ell-2}$ is palindrome. \\
Therefore in both cases $\pi_2$ is palindrome. \\
Let us show that also $\pi_1 = u_{\ell} \: a \: w_{\ell-1}$ is palindrome. \\
Also in this case, by considering the smallest values of $t_{\ell}$ helps us in finding the proof. \\
If $t_{\ell}=0$, then \\
\begin{eqnarray*}
\pi_1 &=& Sym(u_{\ell-1}) \; \lambda \; a \; u_{\ell-1} b \; w_{\ell-2} \; a \; Sym(u_{\ell-1})\\
&=& Sym(u_{\ell-1}) \; a \; u_{\ell-1} \; b \; w_{\ell-2} \; a \; Sym(u_{\ell-1}).
\end{eqnarray*}
The latter is palindrome indeed $Sym(u_{\ell-1})$ is palindrome (since $u_{\ell-1}$ is palindrome) and $u_{\ell-1} \: b \: w_{\ell-2}$ is palindrome. \\
If $t_{\ell}=1$, then 
\[\pi_1=u_{\ell} \; a \; w_{\ell-1} = Sym(u_{\ell-1})\; (a \; w_{\ell-1}) \; a \; w_{\ell-1}=Sym(u_{\ell-1})(a\: w_{\ell-1})^2.\]
If $t_{\ell}=2$, then
\begin{eqnarray*}
\pi_1&=& u_{\ell} \; a \; w_{\ell-1} \\
&=& Sym(u_{\ell-1}) \; (a \; w_{\ell-1})^2 \; a \; w_{\ell-1}\\
&=& Sym(u_{\ell-1}) \; (a \; w_{\ell-1}) \; (a \; w_{\ell-1}) \; (a \; w_{\ell-1}) \\
&=& Sym(u_{\ell-1}) \; (a \; w_{\ell-1})^3.
\end{eqnarray*}
In general, given $i$, 
\[\pi_1=u_{\ell} \; a \; w_{\ell-1} = Sym(u_{\ell-1})\:(a \; w_{\ell-1})^{t_{\ell}} (a \; w_{\ell-1})= Sym(u_{\ell-1})\: (a \; w_{\ell-1})^{t_{\ell+1}}. \]
Therefore, we can prove that $\pi_1$ is palindrome reasoning in an analogous way to what we did in the case of $\pi_2$.
\end{enumerate}

\noindent$(\Rightarrow)$ Let us now prove that every path $w$ that belongs to $F_{\gamma}$ is obtained by \textbf{GenGammaPath}($t$). \\
First of all, if $w$ contains only one peak (resp. valley) then it is $w_0=a^{t_0}b^{t_0}$ (resp. $w_0=b^{t_0}a^{t_0}$) and therefore it has been produced by \textbf{GenGammaPath}($t$). \\
Let us now assume that $w$ contains more than one peak then, for reasons of symmetry, if it has an even (resp. odd) number of peaks then it has a valley (resp. peak) at the center of the path. \\
The two cases are identical, it is only a matter of exchanging the $a$ and $b$ and vice versa. So without loss of generality, we can only consider the case in which there is an odd number of peaks. \\

We are now ready to show that if $w \in F_{\gamma}$ then it is obtained by \textbf{GenGammaPath}($t$). \\
Since $w \in F_{\gamma}$, for Proposition \ref{prop:decomposizione_cammino}, $w$ can be uniquely decomposed into $x\: z \: Sym(x)$ \\
where $\widetilde{z} \in F_{\gamma}$ and $M$ (resp. $M'$) is the highest leftmost (resp. rightmost) point of $w$. \\
Therefore $w=u \: a \: z \: b \: Sym(u)$. By definition, $M$ is preceded by an up step, $a$, and $M'$ is followed by a down step, $b$.
From the Proposition \ref{prop:decomposizione_cammino}, we get that $z$ is a path of $F_{\gamma}$. Since $w$ is obtained from $z$ by attaching at the extremities of $z$, $\Omega =u \: a$ and $Sym(\Omega)= b \: Sym(u)$, then there exists an index $i$ such that $w$ is $w_i$ and that $z$ is $w_ {i-1}$. \\
Since $Sym(u)$ is the symmetric of $u$ as $w \in F_{\gamma}$ by hypothesis, then by definition of element belonging to $F _ {\gamma}$, $w$ is symmetric. Then $w$ has the form of the paths produced by \textbf{GenGammaPath}($t$). \\
To conclude that a generic element of $F_{\gamma}$ is produced by \textbf{GenGammaPath}($t$), it is enough to show that $u$ is of the form $u_i=Sym(u_{i-1})(a\: w_{i-1})^{t_i}$. \\
Since $w_i \in F_{\gamma}$, by $(ii)$ of Proposition \ref{prop:form}, $u \: a \: w_{i-1}$ is palindrome and by $(iv)$ of Proposition \ref{prop:form} $w_{i-1}$ can be decomposed in $w_{i-1}=v_1 \: a \: v_2$, where $v_1$ and $v_2$ are palindromes and $v_2$ is eventually void. \\
$u \: a \: w_{i-1}$ is palindrome, $w_{i-1}$ is an inverted path and then $v_2=Sym(u_{i-1})$.\\
From the fact that $w \in F_{\gamma}$ and that $v \neq \lambda$, it follows that by $(iv)$ of Proposition \ref{prop:form}, $u=v_2\:(a\:v)^t$ with $t \geq 0$. \\
From the fact that $v_2=Sym(u_{i-1})$ and that $u=v_2\:(a\: w_{i-1})^t$, it follows that $u=Sym(u_{i-1})(a \: w_{i-1})^{t_i}$. \\
Therefore every element of $F_{\gamma}$ is obtained by \textbf{GenGammaPath}($t$).
\end{proof}  

\section{Conclusion and further works}
The algorithm \textbf{GenGammaPath}($t$) we have presented in the previous section produces a path $w=w(t)$ on input $t=(t_0, \ldots ,t_{n})$. A possible direction for further research is to use \textbf{GenGammaPath} to obtain a {\em Constant Amortized Time} (CAT) algorithm for the generation of the fixed points of $\gamma$. 
In order to prove that each of these paths is generated in constant amortized time we need to obtain more information about the enumeration of  fixed points of $\gamma$ according to their length. In particular, we expect to determine an explicit formula $p(n)$ that associates to an array $t$ of length $n$ the length of the path $w(t)$ obtained performing \textbf{GenGammaPath}($t$). 
By some preliminary investigation we have that the following recurrence relation holds:
$$\begin{array}{l}
p(0)=2t_0  \\
p(n+1)=p(0)+p(n)+2 \sum_{i=0}^n t_{i+1} ( p(i) + 1) \, .
\end{array}
$$

\end{document}